\documentclass[preprint]{elsarticle}

\usepackage{amssymb}
\usepackage{amsmath}
\usepackage{amsthm}
\usepackage{enumerate}
\usepackage{graphicx}
\usepackage{amsfonts}
\usepackage{minipage}
\usepackage{float}

\newtheorem{example}{Example}

\journal{...}

\newtheorem{theorem}{Theorem}[section]

\newtheorem{lemma}{Lemma}[section]

\newtheorem{cor}{Corollary}[section]

\theoremstyle{definition}

\begin{document}

\begin{frontmatter}

\title{Estimates for the differences of positive linear operators and their derivatives}

\author[1]{Ana-Maria Acu}
\author[2]{Ioan Ra\c sa}
\address[1]{Lucian Blaga University of Sibiu, Department of Mathematics and Informatics, Str. Dr. I. Ratiu, No.5-7, RO-550012  Sibiu, Romania, e-mail: anamaria.acu@ulbsibiu.ro}
\address[2]{Technical University of Cluj-Napoca, Faculty of Automation and Computer Science, Department of Mathematics, Str. Memorandumului nr. 28 Cluj-Napoca, Romania,
	e-mail:  ioan.rasa@math.utcluj.ro }

\begin{abstract}

The present paper deals with the estimate of the differences of certain positive linear operators and their derivatives. Our approach involves operators defined on bounded intervals, as Bernstein operators, Kantorovich operators, genuine Bernstein-Durrmeyer operators, Durrmeyer operators with Jacobi weights. The estimates in quantitative form are given in terms of first modulus of continuity. In order to analyze the theoretical results in the last section we consider some numerical examples.
\end{abstract}

\begin{keyword} first modulus of continuity;  positive linear operators; Bernstein operators;  Durrmeyer operators; Kantorovich operators

	\MSC[2010] 41A25, 41A36.
\end{keyword}

\end{frontmatter}

\section{Introduction}

The de la Vall\'ee Poussin operators of a $2\pi$-periodic integrable function $f$ are defined as 
$$ L_n(f;x)=\dfrac{1}{2\pi}\dfrac{(n!)^2}{(2n)!}2^{2n}\displaystyle\int_{-\pi}^{\pi}f(u)\left(\cos\dfrac{x-u}{2}\right)^{2n}du. $$
These operators are trigonometric analogues of the Bernstein operators. It is well-known that de la Vall\'ee-Poussin operator commutes with the derivative. Indeed, for $f\in C^1_{2\pi}[-\pi,\pi], $
$$ L_n(f;x)=\dfrac{1}{2\pi}\dfrac{(n!)^2}{(2n)!}2^{2n}\displaystyle\int_{-\pi-x}^{\pi-x}f(x+t)\left(\cos\dfrac{t}{2}\right)^{2n}dt $$
and we get
\begin{align*}
\left(L_n(f;x)\right)^{\prime}&=\dfrac{1}{2\pi}\dfrac{(n!)^2}{(2n)!}2^{2n}\left\{
\int_{-\pi-x}^{\pi-x}f^{\prime}(x+t)\left(\cos \dfrac{t}{2}\right)^{2n}dt-f(\pi)\left(\cos\dfrac{\pi-x}{2}\right)^{2n}\right.\\
&\left.+f(-\pi)\left(\cos\dfrac{-\pi-x}{2}\right)^{2n}  \right\}=L_n\left(f^{\prime};x\right).
\end{align*}
Thus $(L_nf)^{(k)}=L_n(f^{(k)})$, for $f\in C^k_{2\pi}[-\pi,\pi]$.
Certainly, this property is not available for the Bernstein operators $B_n$. The polynomials $(B_nf)^{(k)}$ and $B_{n-k}(f^{(k)})$ have degree $n-k$ and converge to $f^{(k)}$. 
This remark motivated us to estimate in terms of moduli of continuity the differences $(L_nf)^{(k)}-L_{n-k}(f^{(k)})$ for certain positive linear operators, as  Bernstein, Kantorovich, genuine Bernstein-Durrmeyer, Durrmeyer operators with Jacobi weights.

The study of differences of certain positive and linear operators has as starting point the problem proposed by Lupa\c s \cite{Adif_A20}, namely the question raised by him was to give an estimate for
$ B_n\circ \overline{\mathbb{B}}_{n}- \overline{\mathbb{B}}_{n}\circ B_n$, where $B_n$ and  $\overline{\mathbb{B}}_n$ are Bernstein operators and Beta operators, respectively. A solution for the problem proposed by Lupa\c s was given for a more general case in \cite{Adif_NAAT}. Some interesting results on this topic were established by Gonska et al. in \cite{Adif_GR1} and \cite{Adif_GR2}. 
 New estimates of the differences of certain operators are provided in a recent paper of Acu et al.  \cite{ana}. These estimates  improve some results concerning the differences of the $U_n^{\rho}$ operators studied in \cite{ARasa, Adif_a}. Very recently,  Aral et al. \cite{AralRasa} obtained some quantitative  results in terms of weighted modulus of continuity for differences of certain positive linear operators defined on unbounded intervals. Also, some estimates for the Chebyshev functional of these operators were provided.
 
 Throught the paper $\|\cdot\|$ denotes the supremum norm and $\omega(f,\cdot)$ is the modulus of continuity of the function $f$.

\section{The Bernstein operators}

Bernstein operators are one of the most important sequences of positive linear operators. These operators were introduced by Bernstein
\cite{Bernstein} and were intensively studied.   For $f\in C[0,1]$, the Bernstein operators  are defined by
\begin{eqnarray*}
	B_n(f,x) = \sum\limits_{k=0}^{n} p_{n,k}(x) \ f\left(\frac{k}{n} \right) , \ \ x \in [0,1],
\end{eqnarray*}
where $p_{n,k}(x)=\displaystyle \binom{n}{k} x^k (1-x)^{n-k}, \ \ k=0,1, \ldots, n. $
\begin{theorem}
	For Bernstein operators the following property holds:
	$$\left\| \left(B_nf\right)^{(r)}-B_{n-r}\left(f^{(r)}\right)\right \|\leq \displaystyle\frac{(r-1)r}{2n}\| f^{(r)}\|+\omega\left(f^{(r)},\frac{r}{n}\right), \,\, f\in C^{r}[0,1],\,\,r=0,1,\dots n.  $$
\end{theorem}
\begin{proof}
	The above differences  can be written as
	\begin{align*}
	&\left(B_n(f;x)\right)^{(r)}-B_{n-r}(f^{(r)}(x))\\
	&=n(n-1)\dots(n-r+1)\displaystyle\sum_{i=0}^{n-r}p_{n-r,i}(x)\Delta_{\frac{1}{n}}^rf\left(\frac{i}{n}\right)-\displaystyle\sum_{i=0}^{n-r}p_{n-r,i}(x)f^{(r)}\left(\frac{i}{n-r}\right)\\
	&=\displaystyle\sum_{i=0}^{n-r}p_{n-r,i}(x)\left\{n(n-1)\dots(n-r+1)\Delta_{\frac{1}{n}}^rf\left(\frac{i}{n}\right)-f^{(r)}\left(\frac{i}{n-r}\right)\right\}\\
	&=\displaystyle\sum_{i=0}^{n-r}p_{n-r,i}(x)\left\{\displaystyle\frac{n(n-1)\dots(n-r+1)}{n^r}r!\left[\frac{i}{n},\dots,\frac{i+r}{n};f\right]-f^{(r)}\left(\frac{i}{n-r}\right)\right\}\\
	&=\displaystyle\sum_{i=0}^{n-r}p_{n-r,i}(x)\left\{\displaystyle\frac{n(n-1)\dots(n-r+1)}{n^r}f^{(r)}(\xi_i)-f^{(r)}\left(\frac{i}{n-r}\right)\right\}\\
	&=\displaystyle\sum_{i=0}^{n-r}p_{n-r,i}(x)\left\{\left(\displaystyle\frac{n(n-1)\dots(n-r+1)}{n^r}-1\right)f^{(r)}(\xi_i)+f^{(r)}(\xi_i)-f^{(r)}\left(\frac{i}{n-r}\right)\right\},
	\end{align*}
	where $\displaystyle\frac{i}{n}\leq\xi_i\leq\frac{i+r}{n}$.\\
	We have,
	$$ 0\leq 1-\displaystyle\frac{n(n-1)\dots(n-r+1)}{n^r}\leq\frac{r(r-1)}{2n}\, \textrm{ and }\, \displaystyle\frac{i}{n}\leq\frac{i}{n-r}\leq\frac{i+r}{n},\textrm{ for } 0\leq i\leq n-r. $$
	Therefore,
	$$\left\| \left(B_nf\right)^{(r)}-B_{n-r}\left(f^{(r)}\right)\right \|\leq \displaystyle\frac{(r-1)r}{2n}\| f^{(r)}\|+\omega\left(f^{(r)},\frac{r}{n}\right).$$
\end{proof}

\section{The Kantorovich operators}
These operators are the integral modification of Bernstein operators   and were introduced by Kantorovich \cite{K} as follows
\begin{equation}\label{Y} K_{n}(f;x)=(n+1)\displaystyle\sum_{k=0}^{n}p_{n,k}(x)\int_{\frac{k}{n+1}}^{\frac{k+1}{n+1}}f(t)dt,\,\, f\in L_{1}[0,1].  \end{equation}
The Kantorovich operators are related to the Bernstein polynomials by:
$$K_n(f;x)=\left[B_{n+1}(F;x)\right]^{\prime}, \textrm{ where } F(x)=\displaystyle\int_0^xf(t)dt.  $$
\begin{theorem}
	The Kantorovich operators verify
	$$\left\| \left(K_nf\right)^{(r)}-K_{n-r}\left(f^{(r)}\right)\right \|\leq \displaystyle\frac{(r+1)r}{2(n+1)}\| f^{(r)}\|+\omega\left(f^{(r)},\frac{r+1}{n+1}\right), \,\, f\in C^{r}[0,1],\,\,r=0,1,\dots n.  $$
\end{theorem}
\begin{proof} The $r^{th}$ derivative of Kantorovich polynomials can be written as follows:
	\begin{align*}
	\left(K_n(f;x)\right)^{(r)}&=\left(B_{n+1}F(x)\right)^{(r+1)}=\sum_{i=0}^{n-r}p_{n-r,i}(x)\left\{(n+1)n(n-1)\dots(n-r+1)\Delta_{\frac{1}{n+1}}^{r+1}F\left(\frac{i}{n+1}\right)\right\}\\
	&=\displaystyle\sum_{i=0}^{n-r}p_{n-r,i}(x)\frac{(n+1)n(n-1)\dots (n-r+1)}{(n+1)^{r+1}}(r+1)!\left[\frac{i}{n+1},\dots,\frac{i+r+1}{n+1};F\right]\\
	&=\displaystyle\sum_{i=0}^{n-r}p_{n-r,i}(x)\frac{(n+1)n(n-1)\dots(n-r+1)}{(n+1)^{r+1}}F^{(r+1)}(\xi_i)\\
	&=\displaystyle\sum_{i=0}^{n-r}p_{n-r,i}(x)\frac{(n+1)n(n-1)\dots (n-r+1)}{(n+1)^{r+1}}f^{(r)}(\xi_i).
	\end{align*}
For the differences of Kantorovich operators  we obtain
	\begin{align*}
	&\left(K_n(f;x)\right)^{(r)}-K_{n-r}\left(f^{(r)}(x)\right)\\
	&=\displaystyle\sum_{i=0}^{n-r}p_{n-r,i}(x)\frac{n(n-1)\dots(n-r+1)}{(n+1)^r}f^{(r)}(\xi_i)-\displaystyle\sum_{i=0}^{n-r}(n-r+1)p_{n-r,i}(x)\int_{\frac{i}{n-r+1}}^{\frac{i+1}{n-r+1}}f^{(r)}(t)dt\end{align*}
	\begin{align*}
	&=\displaystyle\sum_{i=0}^{n-r}p_{n-r,i}(x)\frac{n(n-1)\dots(n-r+1)}{(n+1)^r}f^{(r)}(\xi_i)-\displaystyle\sum_{i=0}^{n-r}p_{n-r,i}(x)f^{(r)}(\eta_i)\\
	&=\displaystyle\sum_{i=0}^{n-r}p_{n-r,i}(x)\left\{\left(\frac{n(n-1)\dots(n-r+1)}{(n+1)^r}-1\right)f^{(r)}(\xi_i)+f^{(r)}(\xi_i)-f^{(r)}(\eta_i)\right\},
	\end{align*}
	where $\displaystyle\frac{i}{n+1}\leq\xi_i\leq\frac{i+r+1}{n+1}$ and $\displaystyle\frac{i}{n-r+1}\leq\eta_i\leq\frac{i+1}{n-r+1}$.\\
	Let us remark that
	$$ 0\leq 1-\displaystyle\frac{n(n-1)\dots(n-r+1)}{(n+1)^r}\leq\frac{r(r+1)}{2(n+1)}.$$
	Therefore,
	$$\left\| \left(K_nf\right)^{(r)}-K_{n-r}\left(f^{(r)}\right)\right \|\leq \displaystyle\frac{(r+1)r}{2(n+1)}\| f^{(r)}\|+\omega\left(f^{(r)},\frac{r+1}{n+1}\right).$$
\end{proof}
In order to extend the above result we will define
 the operator
$$Q_n^kf:=\displaystyle\frac{n^k(n-k)!}{n!}\left(B_n(f^{(-k)})\right)^{(k)},\,\,f\in C[0,1],$$
where $f^{(-k)}$ is an antiderivative of order $k$ for the function $f$.

\begin{theorem}
	For the  operators $Q_n^k$ the following property holds:
	$$\left\| \left(Q_n^kf\right)^{(r)}-Q_{n-r}^k\left(f^{(r)}\right)\right \|\leq \displaystyle\frac{(2k+r-1)r}{2n}\| f^{(r)}\|+\omega\left(f^{(r)},\frac{k+r}{n}\right), \,\, f\in C^{r}[0,1],\,\,r=0,1,\dots n.  $$
\end{theorem}
\begin{proof} The above inequality follows from
	\begin{align*}
	 \left|(Q_n^kf)^{(r)}-Q_{n-r}^k(f^{(r)})\right|&=\left| \displaystyle\frac{n^k(n-k)!}{n!}\left(B_n(f^{(-k)})\right)^{(k+r)}-\frac{(n-r)^k(n-r-k)!}{(n-r)!}\left(B_{n-r}f^{(r-k)}\right)^{(k)} \right|\\
	& =\left| \displaystyle\frac{n^k(n-k)!}{n!}\frac{n!}{(n-k-r)!}\sum_{i=0}^{n-k-r}p_{n-k-r,i}\Delta_{\frac{1}{n}}^{k+r}f^{(-k)}\left(\frac{i}{n}\right)\right.\\
	&\left.
-	\displaystyle\frac{(n-r)^k(n-r-k)!}{(n-r)!}\frac{(n-r)!}{(n-k-r)!}\sum_{i=0}^{n-k-r}p_{n-k-r,i}\Delta_{\frac{1}{n-r}}^{k}f^{(r-k)}\left(\frac{i}{n-r}\right) \right|\\
&=\left|\displaystyle\sum_{i=0}^{n-k-r}p_{n-k-r,i}\left\{\displaystyle\frac{n^k(n-k)!}{(n-k-r)!}\frac{(k+r)!}{n^{k+r}}\left[\frac{i}{n},\dots,\frac{i+k+r}{n};f^{(-k)}\right]\right.\right.\\
&\left.\left. -(n-r)^k\frac{k!}{(n-r)^k}\left[\frac{i}{n-r},\dots,\frac{i+k}{n-k}; f^{(r-k)}\right]\right\}\right|\\
&=\left|\displaystyle\sum_{i=0}^{n-k-r}p_{n-k-r,i}\left\{\frac{(n-k)!}{(n-k-r)!n^r}f^{(r)}(\xi_i)-f^{(r)}(\eta_i)\right\}\right|\\
&\leq\displaystyle\sum_{i=0}^{n-k-r}\!p_{n\!-\!k\!-\!r,i}\left|\left(\frac{(n\!-\!k)\dots(n\!-\!k\!-\!r\!+\!1)}{n^r}\!-\!1\right)f^{(r)}(\xi_i)\!+\!f^{(r)}(\xi_i)\!-\!f^{(r)}(\eta_i)\right|\end{align*}
\begin{align*}
&\leq \left|\frac{(n-k)\dots (n-k-r+1)}{n^r}-1\right|\,\left\|f^{(r)}\right\|+\omega\left(f^{(r)},\frac{k+r}{n}\right)\\
&\leq \displaystyle\frac{(2k+r-1)r}{2n}\,\left\|f^{(r)}\right\|+\omega\left(f^{(r)},\frac{k+r}{n}\right),
	\end{align*}
	where $\displaystyle\frac{i}{n}\leq \xi_i\leq \frac{i+k+r}{n}$ and $\displaystyle\frac{i}{n-r}\leq\eta_i\leq\frac{i+k}{n-r}$.
	\end{proof}

\section{The Durrmeyer operators with Jacobi weights}

The classical Durrmeyer operators are the integral modification of Bernstein operators so as to approximate Lebesgue integrable functions defined on the interval $[0,1]$. These operators were introduced by Durrmeyer \cite{1a} and, independently, by Lupa\c s \cite{dif_A19} and are defined as
 \begin{eqnarray}\label{X}
 M_n(f,x) =(n+1) \sum\limits_{k=0}^{n} p_{n,k}(x)  \int\limits_{0}^{1} p_{n,k}(t) \ f(t) \, dt , \ \ x \in [0,1].
 \end{eqnarray}
  
 Let $w^{(\alpha,\beta)}(x)=x^{\alpha}(1-x)^{\beta},\,\alpha,\beta>-1$ be a Jacobi weight function on the interval $(0,1)$ and $L_{p}^{w^{(\alpha,\beta)}}[0,1]$ be the space of Lebesgue-measurable functions $f$ on $[0,1]$ for which the weighted $L_p$-norm is finite.   The Durrmeyer operators can be generalized as follows
 $$ M_n^{(\alpha,\beta)}=\displaystyle\sum_{k=0}^n p_{n,k}(x)\displaystyle\frac{1}{c_{n,k}^{(\alpha,\beta)}}\int_0^1p_{n,k}(t)w^{(\alpha,\beta)}(t)f(t)dt, $$
 where $c_{n,k}^{(\alpha,\beta)}=\displaystyle\int_0^1p_{n,k}(t)w^{(\alpha,\beta)}(t)dt$ and $f\in L_1^{(\alpha,\beta)}[0,1]$. See \cite{D11} and \cite{Paltanea}.
 
 The classical  Durrmeyer operators $M_n$ are obtained for $\alpha=\beta=0$.

In order to give the estimate for the difference of the Durrmeyer
operators  we need the following result (see, e.g., \cite{ana}):
\begin{lemma}
	Let $F:C[0,1]\to\mathbb{R}$ be a positive linear functional
	with $F(e_0)=1$ and $F(e_1)=b$. Then, for
	each $\varphi\in C^2[0,1]$ there is $\xi\in[0,1]$ such that
	$$ F(\varphi)-\varphi(b)=\left(F(e_2)-e_2(b)\right)\displaystyle\frac{\varphi^{\prime\prime}(\xi)}{2}, $$
	where $e_r(x)=x^r,\, r=0,1,\dots$.
\end{lemma}

	Let $\varphi\in C^2[0,1]$. With fixed $0\leq r\leq n$ and $0\leq k\leq n-r$, consider the functional
\begin{align*}A_{n,k}^{(\alpha,\beta)}(\varphi)&:=\displaystyle (n+\alpha+\beta+r+1)\int_0^1 p_{n+\beta+\alpha+r,k+\alpha+r}(t)\varphi(t)dt\\
&- (n+\alpha+\beta-r+1)\int_0^1p_{n-r+\alpha+\beta,k+\alpha}(t)\varphi(t)dt
=B_{n,k}^{(\alpha,\beta)}(\varphi)-C_{n,k}^{(\alpha,\beta)}(\varphi),  \end{align*}
where
	\begin{align*}
&B_{n,k}^{(\alpha,\beta)}(\varphi)=(n+\alpha+\beta+r+1)\int_0^1p_{n+\beta+\alpha+r,k+\alpha+r}(t)\varphi(t)dt,\\
&C_{n,k}^{(\alpha,\beta)}(\varphi)=(n+\alpha+\beta-r+1)\int_0^1p_{n-r+\alpha+\beta,k+\alpha}(t)\varphi(t)dt.
\end{align*}

\begin{lemma}\label{l1}
	The functional $A_{n,k}^{(\alpha,\beta)}$ verifies
	$$ |A_{n,k}^{(\alpha,\beta)}(\varphi)|\leq \displaystyle\frac{1}{4}\|\varphi^{\prime\prime}\|\frac{n+\alpha+\beta+3}{(n+\alpha+\beta+3)^2-r^2}+\omega\left(\varphi,\frac{r(n-r+|\beta-\alpha|)}{(n+2+\alpha+\beta)^2-r^2}\right),  $$
	where $\omega$ is the first order modulus of
	continuity.
\end{lemma}
\begin{proof}
	By simple calculations, we get
\begin{align*}
& B_{n,k}^{(\alpha,\beta)}(e_0)=1,\, B_{n,k}^{(\alpha,\beta)}(e_1)=\displaystyle\frac{k+\alpha+r+1}{n+r+\alpha+\beta+2},\, B_{n,k}^{(\alpha,\beta)}(e_2)=\displaystyle\frac{(k\!+\!\alpha\!+\!r\!+\!1)(k\!+\!\alpha\!+\!r\!+\!2)}{(n\!+\!r\!+\!\alpha\!+\!\beta\!+\!2)(n\!+\!r\!+\!\alpha\!+\!\beta\!+\!3)},\\
&C_{n,k}^{(\alpha,\beta)}(e_0)=1,\, C_{n,k}^{(\alpha,\beta)}(e_1)=\frac{k+\alpha+1}{n-r+\alpha+\beta+2},\, C_{n,k}^{(\alpha,\beta)}(e_2)=\displaystyle\frac{(k\!+\!\alpha\!+\!1)(k\!+\!\alpha\!+\!2)}{(n\!-\!r\!+\!\alpha\!+\!\beta\!+\!2)(n\!-\!r\!+\!\alpha\!+\!\beta\!+\!3)}.
\end{align*}
Therefore,
\begin{align*}
&|A_{n,k}^{(\alpha,\beta)}(\varphi)|=|B_{n,k}^{(\alpha,\beta)}(\varphi)-C_{n,k}^{(\alpha,\beta)}(\varphi)|\leq\left|B_{n,k}^{(\alpha,\beta)}(\varphi)-\varphi\left(\frac{k+\alpha+r+1}{n+r+\alpha+\beta+2}\right)\right|\\
&+\left|C_{n,k}^{(\alpha,\beta)}(\varphi)-\varphi\left(\frac{k+\alpha+1}{n-r+\alpha+\beta+2}\right)\right|
+\left|\varphi\left(\frac{k+\alpha+r+1}{n+r+\alpha+\beta+2}\right)-\varphi\left(\frac{k+\alpha+1}{n-r+\alpha+\beta+2}\right)\right|\\
&\leq\displaystyle\frac{1}{2}\|\varphi^{\prime\prime}\|\left(\frac{(k+\alpha+r+1)(\beta+1+n-k)}{(2+\alpha+\beta+n+r)^2(\alpha+\beta+n+r+3)}+\frac{(k+\alpha+1)(\beta+1+n-r-k)}{(2+\alpha+\beta+n-r)^2(\alpha+\beta+n-r+3)}\right)\\
&+\omega\left(\varphi,\frac{r|n-r+\beta-\alpha-2k|}{(2+\alpha+\beta+n+r)(2+\alpha+\beta+n-r)}\right)\\
&\leq \displaystyle\frac{1}{4}\|\varphi^{\prime\prime}\|\frac{n+\alpha+\beta+3}{(n+\alpha+\beta+3)^2-r^2}+\omega\left(\varphi,\frac{r(n-r+|\beta-\alpha|)}{(n+2+\alpha+\beta)^2-r^2}\right).
\end{align*}
\end{proof}

\begin{theorem}\label{t2A}
	For Durrmeyer operators  with Jacobi weights the following property holds:
	\begin{align*}
	&\left\|\displaystyle\frac{\Gamma(n+\alpha+\beta+r+2)\Gamma(n-r+1)}{\Gamma(n+\alpha+\beta+2)\Gamma(n+1)}\left(M_n^{(\alpha,\beta)}f\right)^{(r)}-M_{n-r}^{(\alpha,\beta)}\left(f^{(r)}\right)\right\|\\
	&\leq \displaystyle\frac{1}{4}\|f^{(r+2)}\|\frac{n+\alpha+\beta+3}{(n+\alpha+\beta+3)^2-r^2}+\omega\left(f^{(r)},\frac{r(n-r+|\beta-\alpha|)}{(n+2+\alpha+\beta)^2-r^2}\right), \end{align*}
	where $f\in C^{r+2}[0,1],\,\,r=1,\dots, n. $
\end{theorem}
\begin{proof}
		In \cite{Abel}, Abel et al. proved the following identity for the derivatives of $M_n^{(\alpha,\beta)}f$:
	\begin{equation}\label{Abel} \left(M_n^{(\alpha,\beta)}f\right)^{(r)}=\displaystyle\frac{n(n-1)\dots (n-r+1)}{(n+\alpha+\beta+2)(n+\alpha+\beta+3)\dots (n+\alpha+\beta+r+1)} M_{n-r}^{(\alpha+r,\beta+r)}f^{(r)},\end{equation}
	where $f^{(r)}\in L_1^{w^{(\alpha+r,\beta+r)}}[0,1]$ and $r\leq n$.
	
	By simple calculations it can be shown that
	$$ \left( M_n^{(\alpha,\beta)}(f;x)\right)^{(r)}=\displaystyle\frac{\Gamma(n+\alpha+\beta+2)\Gamma(n+1)}{\Gamma(n+\alpha+\beta+r+1)\Gamma(n-r+1)}\sum_{k=0}^{n-r}p_{n-r,k}(x)\int_0^1p_{n+\beta+\alpha+r,k+\alpha+r}(t)f^{(r)}(t)dt.  $$
	We can write
	\begin{align*}
	&\displaystyle\frac{\Gamma(n+\alpha+\beta+r+2)\Gamma(n-r+1)}{\Gamma(n+\alpha+\beta+2)\Gamma(n+1)}\left(M_n^{(\alpha,\beta)}(f;x)\right)^{(r)}-M_{n-r}^{(\alpha,\beta)}\left(f^{(r)};x\right)\\
	&=(n+\alpha+\beta+r+1)\sum_{k=0}^{n-r}p_{n-r,k}(x)\int_0^1p_{n+\beta+\alpha+r,k+\alpha+r}(t)f^{(r)}(t)dt\\
	&-\displaystyle\sum_{k=0}^{n-r}p_{n-r,k}(x)\frac{1}{\int_0^1p_{n-r,k}(t)t^{\alpha}(1-t)^\beta dt}\int_0^1p_{n-r,k}(t)t^{\alpha}(1-t)^{\beta}f^{(r)}(t)dt\\
	&=\displaystyle\sum_{k=0}^{n-r}p_{n-r,k}(x)\left\{(n+\alpha+\beta+r+1)\int_0^1 p_{n+\beta+\alpha+r,k+\alpha+r}(t)f^{(r)}(t)dt\right.\\
	&-\left. (n+\alpha+\beta-r+1)\int_0^1p_{n-r+\alpha+\beta,k+\alpha}(t)f^{(r)}(t)dt\right\}=\sum_{k=0}^{n-r}p_{n-r,k}(x)A_{n,k}^{(\alpha,\beta)}(f^{(r)}).
	\end{align*}
	Using Lemma \ref{l1} for $f\in C^{r+2}[0,1]$, the proof is completed.
	\end{proof}

\begin{theorem}
	Let $f\in C^{r+2}[0,1]$, $r=0,1,\dots,n$. The Durrmeyer operators with Jacobi weights verify
	\begin{align*}  \| (M_n^{(\alpha,\beta)}f)^{(r)}-M_{n-r}^{(\alpha,\beta)}(f^{(r)}) \| &\leq \displaystyle\frac{r(\alpha+\beta+r+1)}{n+\alpha+\beta+2}\| f^{(r)}\|+
	\displaystyle\frac{1}{4}\|f^{(r+2)}\|\frac{n+\alpha+\beta+3}{(n+\alpha+\beta+3)^2-r^2}\\
	&+\omega\left(f^{(r)},\frac{r(n-r+|\beta-\alpha|)}{(n+2+\alpha+\beta)^2-r^2}\right). \end{align*}
\end{theorem}
\begin{proof}
	Since $\left(\displaystyle\frac{\| f^{(r)}\|}{r!}e_r\pm f\right)^{(r)}\geq 0$, using the relation (\ref{Abel}), we get $$\left(M_n\left(\displaystyle \frac{\|f^{(r)}\|}{r!}e_r\pm f\right)\right)^{(r)}\geq 0.$$
	It is well-known, see	\cite[p.7]{15}, that:
	$$  M_n^{(\alpha,\beta)}(e_r;x)=\displaystyle\frac{1}{(n+\alpha+\beta+2)^{\overline{r}}}\sum_{k=0}^r {r\choose k}n^{\underline{j}}(\alpha+r)^{\overline{r-k}}x^k,$$
	where $x^{\overline{l}}$ and $x^{\underline{l}}$, $x\in\mathbb{R}$ are the rising and falling factorials respectively, given by $x^{\overline{l}}=\prod_{\nu=0}^{l-1}(x+\nu)$, $x^{\underline{l}}=\prod_{\nu=0}^{l-1}(x-\nu)$ for $l\in\mathbb{N}$, $x^{\overline{0}}=x^{\underline{0}}=1$.
	
	Thus, $\displaystyle\frac{\| f^{(r)}\|}{r!}\frac{\Gamma(n+1)\Gamma(n+\alpha+\beta+2)}{\Gamma(n-r+1)\Gamma(n+\alpha+\beta+r+2)}r!\pm (M_n^{(\alpha,\beta)}f)^{(r)}\geq 0,$ i.e.,
	$$ \| \left(M_n f\right)^{(r)}\|\leq\frac{\Gamma(n+1)\Gamma(n+\alpha+\beta+2)}{\Gamma(n-r+1)\Gamma(n+\alpha+\beta+r+2)} \|f^{(r)}\|.$$
	Denote
	$$ \theta(f;n,r)=\left\|\displaystyle\frac{\Gamma(n+\alpha+\beta+r+2)\Gamma(n-r+1)}{\Gamma(n+\alpha+\beta+2)\Gamma(n+1)}\left(M_n^{(\alpha,\beta)}f\right)^{(r)}-M_{n-r}^{(\alpha,\beta)}\left(f^{(r)}\right)\right\|. $$
	The differences of Durrmeyer operators with Jacobi weights  can be written as
	\begin{align*}
	&\|(M_n^{(\alpha,\beta)}f)^{(r)}-M_{n-r}^{(\alpha,\beta)}(f^{(r)})\|\\
	&\leq \left|\displaystyle\frac{\Gamma(n-r+1)\Gamma(n+\alpha+\beta+r+2)}{\Gamma(n+1)\Gamma(n+\alpha+\beta+2)}-1\right|\,\|(M_n^{(\alpha,\beta)}f)^{(r)}\|
	+\theta(f;n,r)\\
	&\leq\left(1-\frac{n(n-1)\dots(n-r+1)}{(n+\alpha+\beta+r+1)(n+\alpha+\beta+r)\dots(n+\alpha+\beta+2)}\right)\| f^{(r)}\|+\theta(f;n,r)\\
	&\leq \left(1-\left(\displaystyle\frac{n-r+1}{n+\alpha+\beta+2}\right)^r\right)\| f^{(r)}\|+\theta(f;n,r)
	\leq \left(1-\displaystyle\frac{n-r+1}{n+\alpha+\beta+2}\right)r\| f^{(r)}\|+\theta(f;n,r)\\
	&= \displaystyle\frac{r(r+\alpha+\beta+1)}{n+\alpha+\beta+2}\| f^{(r)}\|+\theta(f;n,r).
	\end{align*}
	Using Theorem \ref{t2A} the proof is completed.
	
\end{proof}

\begin{cor}\label{tt2}
	For Durrmeyer operators  the following property holds:
	$$\left\|\displaystyle\frac{(n\!+\!r\!+\!1)!(n-r)!}{(n\!+\!1)!n!}\left(M_nf\right)^{(r)}\!-\!M_{n-r}(f^{(r)})\right\|\!\leq\! \displaystyle\frac{1}{4}\|f^{(r+2)}\|\frac{n+3}{(n+3)^2\!-\!r^2}+\omega\left(f^{(r)},\frac{r(n-r)}{(n+2)^2-r^2}\right),  $$
	where $f\in C^{r+2}[0,1],\,\,r=1,\dots, n. $
\end{cor}

\begin{cor}
	Let $f\in C^{r+2}[0,1]$, $r=0,1,\dots,n$. The Durrmeyer operators verify
	\begin{align*}  \| (M_nf)^{(r)}-M_{n-r}(f^{(r)}) \|&\leq \displaystyle\frac{r(r+1)}{n+2}\| f^{(r)}\|+\frac{n+3}{4[(n+3)^2-r^2]}\| f^{(r+2)}\|+\omega\left(f^{(r)},\displaystyle\frac{r(n-r)}{(n+2)^2-r^2}\right). \end{align*}
\end{cor}

\section{The genuine Bernstein-Durrmeyer operators}
 The genuine Bernstein-Durrmeyer operators (see  \cite{dif_A2}, \cite{dif_GBD}) are defined as follows:
$$ U_{n}(f;x)\!=\!(1\!-\!x)^n f(0)\!+ x^n f(1)\!+\!(n\!-\!1)
\displaystyle\sum_{k=1}^{n-1}\left(\int_{0}^1\!\! f(t)p_{n-2,k-1}(t)dt\!\right)p_{n,k}(x),\quad  f\!\in\! C[0,1].  $$

These operators are limits of the Bernstein-Durrmeyer operators with Jacobi weights (see \cite{D11}, \cite{D12}, \cite{Paltanea}), namely
$$  U_nf=\displaystyle\lim_{\alpha\to-1, \beta\to -1} M_n^{(\alpha,\beta)}f. $$

\begin{theorem}\label{t2}
	For the genuine Bernstein-Durrmeyer operators  the following property holds:
	\begin{align*}\left\|\displaystyle\frac{(n+r-1)!(n-r)!}{(n-1)!n!}\left(U_{n}f\right)^{(r)}-U_{n-r}(f^{(r)})\right\|&\leq \displaystyle\frac{1}{4} \frac{n+1}{(n+1)^2\!-\!r^2}\|f^{(r+2)}\|+\frac{r}{n+r} \| f^{(r+1)}\|\\
	&+\omega\left(f^{(r)},\frac{r(n-2-r)}{n^2-r^2}\right),
\end{align*}
	where $f\in C^{r+2}[0,1],\,\,r=1,\dots, n-2. $
\end{theorem}

\begin{proof} First,
\begin{align*}
&\displaystyle\frac{(n+r-1)!(n-r)!}{(n-1)!n!}\left(U_{n}(f;x)\right)^{(r)}-U_{n-r}\left(f^{(r)};x\right)
=\displaystyle\frac{(n+r-1)!(n-r)!}{(n-1)!n!} \frac{n(n-1)\dots(n-r+1)}{(n\!+\!r\!-\!2)(n+r-3)\dots n}\\ &\displaystyle\sum_{k=0}^{n-r}p_{n-r,k}(x)\int_0^1p_{n+r-2,k+r-1}(t)f^{(r)}(t)dt\!-\!U_{n-r}\left(f^{(r)};x\right)\\
&=(n+r-1)\displaystyle\sum_{k=0}^{n-r}p_{n-r,k}(x)\int_0^1p_{n+r-2,k+r-1}(t)f^{(r)}(t)dt-p_{n-r,0}(x)f^{(r)}(0)-p_{n-r,n-r}(x)f^{(r)}(1)\\
&-(n-r-1)\displaystyle\sum_{k=1}^{n-r-1}p_{n-r,k}(x)\int_0^1p_{n-r-2,k-1}(t)f^{(r)}(t)dt\\
&=(n\!+\!r\!-\!1)p_{n-r,0}(x)\int_0^1p_{n+r-2,r-1}(t)f^{(r)}(t)dt+(n+r-1)p_{n-r,n-r}(x)\int_0^1 p_{n+r-2,n-1}(t) f^{(r)}(t)dt\\
&-p_{n-r,0}(x)f^{(r)}(0)-p_{n-r,n-r}(x)f^{(r)}(1)+(n+r-1)\sum_{k=1}^{n-r-1}p_{n-r,k}(x)\int_0^1p_{n+r-2,k+r-1}(t)f^{(r)}(t)dt\\
&-(n-r-1)\sum_{k=1}^{n-r-1}p_{n-r,k}(x)\int_0^1p_{n-r-2,k-1}(t)f^{(r)}(t)dt\\
&=p_{n-r,0}(x)\int_0^1(n+r-1)p_{n+r-2,r-1}(t)\left(f^{(r)}(t)-f^{(r)}(0)\right)dt\\
&+p_{n-r,n-r}(x)\int_0^1(n+r-1)p_{n+r-2,n-1}(t)\left(f^{(r)}(t)-f^{(r)}(1)\right)dt\\
&+\displaystyle\sum_{k=1}^{n-r-1}p_{n-r,k}(x)\int_0^1\left[(n+r-1)p_{n+r-2,k+r-1}(t)-(n-r-1)p_{n-r-2,k-1}(t)\right]f^{(r)}(t)dt.
\end{align*}
It follows that
\begin{align*}
&\left|\displaystyle\int_0^1(n+r-1)p_{n+r-2,r-1}(t)\left(f^{(r)}(t)-f^{(r)}(0)\right)dt\right|\leq\int_0^1(n+r-1)p_{n+r-2,r-1}(t)\|f^{(r+1)}\|t dt\\
&=\| f^{(r+1)}\|\frac{r}{n+r},\\
&\left|\displaystyle\int_0^1(n+r-1)p_{n+r-2,n-1}(t)\left(f^{(r)}(t)-f^{(r)}(1)\right)dt\right|\leq \| f^{(r+1)}\|\frac{r}{n+r}.
\end{align*}
Using Lemma \ref{l1}, we get
\begin{align*}
&\int_0^1\left[(n+r-1)p_{n+r-2,k+r-1}(t)-(n-r-1)p_{n-r-2,k-1}(t)\right]f^{(r)}(t)dt=A_{n-2,k-1}(f^{(r)})\\
&\leq \displaystyle \frac{1}{4}\| f^{(r+2)}\|\frac{n+1}{(n+1)^2-r^2}+\omega\left(f^{(r)},\frac{r(n-2-r)}{n^2-r^2}\right).
\end{align*}
Since $\displaystyle\sum_{k=1}^{n-r-1}p_{n-r,k}(x)\leq 1$, $p_{n-r,0}(x)+p_{n-r,n-r}(x)\leq 1$, the proof is completed.	
\end{proof}

\section{Numerical Results}
In this section 
we will give some numerical examples in order to show the relevance of the theoretical results.

\begin{example} Let $f(x)=\displaystyle\frac{1}{32\pi}\left\{4\pi x\cos(2\pi x)-\pi\cos(2\pi x)-6\sin(2\pi x)\right\}$, $r=3$ and $E_{n,r}(f;x)=\left|\left(B_n(f;x)\right)^{(r)}-B_{n-r}(f^{(r)}(x))\right|$. 
	In Figure \ref{fig:1} are given the graphs of the functions $f^{(r)}$,  $B_{n-r}(f^{(r)})$ and  $(B_nf)^{(r)}$ for $n=50$ and $r=3$.
 Also,   for $n\in\{50,100,150\}$ the absolute value of the differences  are illustrated in Figure \ref{fig:2}.
	
	\begin{minipage}{\linewidth}
		\centering
		\begin{minipage}{0.4\linewidth}
			\begin{figure}[H]
				\includegraphics[width=\linewidth]{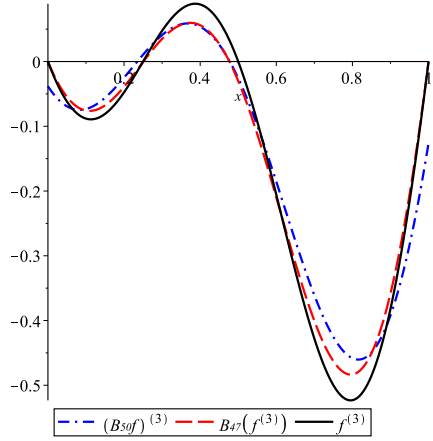}
				\caption{Approximation process by  $B_{n-r}(f^{(r)})$ and $(B_nf)^{(r)}$ }\label{fig:1}
			\end{figure}
		\end{minipage}
		\hspace{0.05\linewidth}
		\begin{minipage}{0.4\linewidth}
			\begin{figure}[H]
				\includegraphics[width=\linewidth]{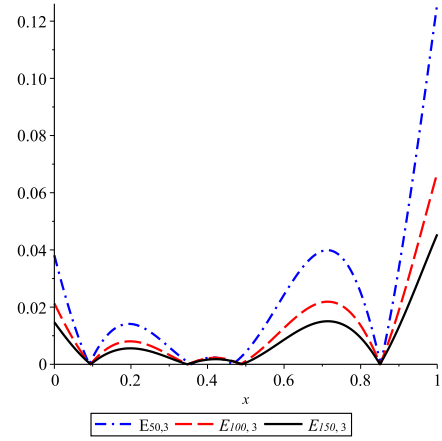}
				\caption{Error $E_{n,r}(f;x)$, for  $n\in\{50,100,150\}$}\label{fig:2}
			\end{figure}
		\end{minipage}
	\end{minipage}
\end{example}

\begin{example} Let $f(x)=\displaystyle -\frac{1}{4\pi^2}\sin(2\pi x)-\frac{32}{\pi^2}\sin\left(\frac{1}{4}\pi x\right)$, $r=2$ and $E_{n,r}(f;x)=$ \linebreak $\left|\left(K_n(f;x)\right)^{(r)}-K_{n-r}(f^{(r)}(x))\right|$. 
	In Figure \ref{fig:3} are given the graphs of the functions $f^{(r)}$,  $K_{n-r}(f^{(r)})$ and  $(K_nf)^{(r)}$ for $n=50$ and $r=2$.
	Also,   for $n\in\{50,100,150\}$ the absolute value of the differences  are illustrated in Figure \ref{fig:4}.
	
	\begin{minipage}{\linewidth}
		\centering
		\begin{minipage}{0.4\linewidth}
			\begin{figure}[H]
				\includegraphics[width=\linewidth]{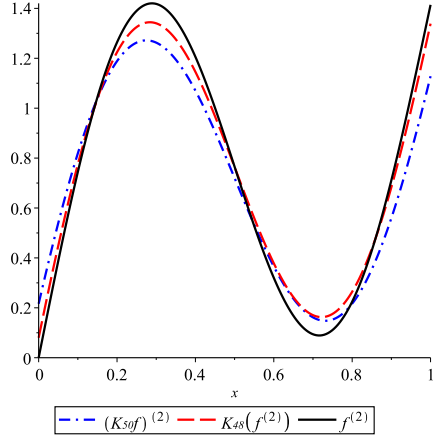}
				\caption{Approximation process by  $K_{n-r}(f^{(r)})$ and $(K_nf)^{(r)}$ }\label{fig:3}
			\end{figure}
		\end{minipage}
		\hspace{0.05\linewidth}
		\begin{minipage}{0.4\linewidth}
			\begin{figure}[H]
				\includegraphics[width=\linewidth]{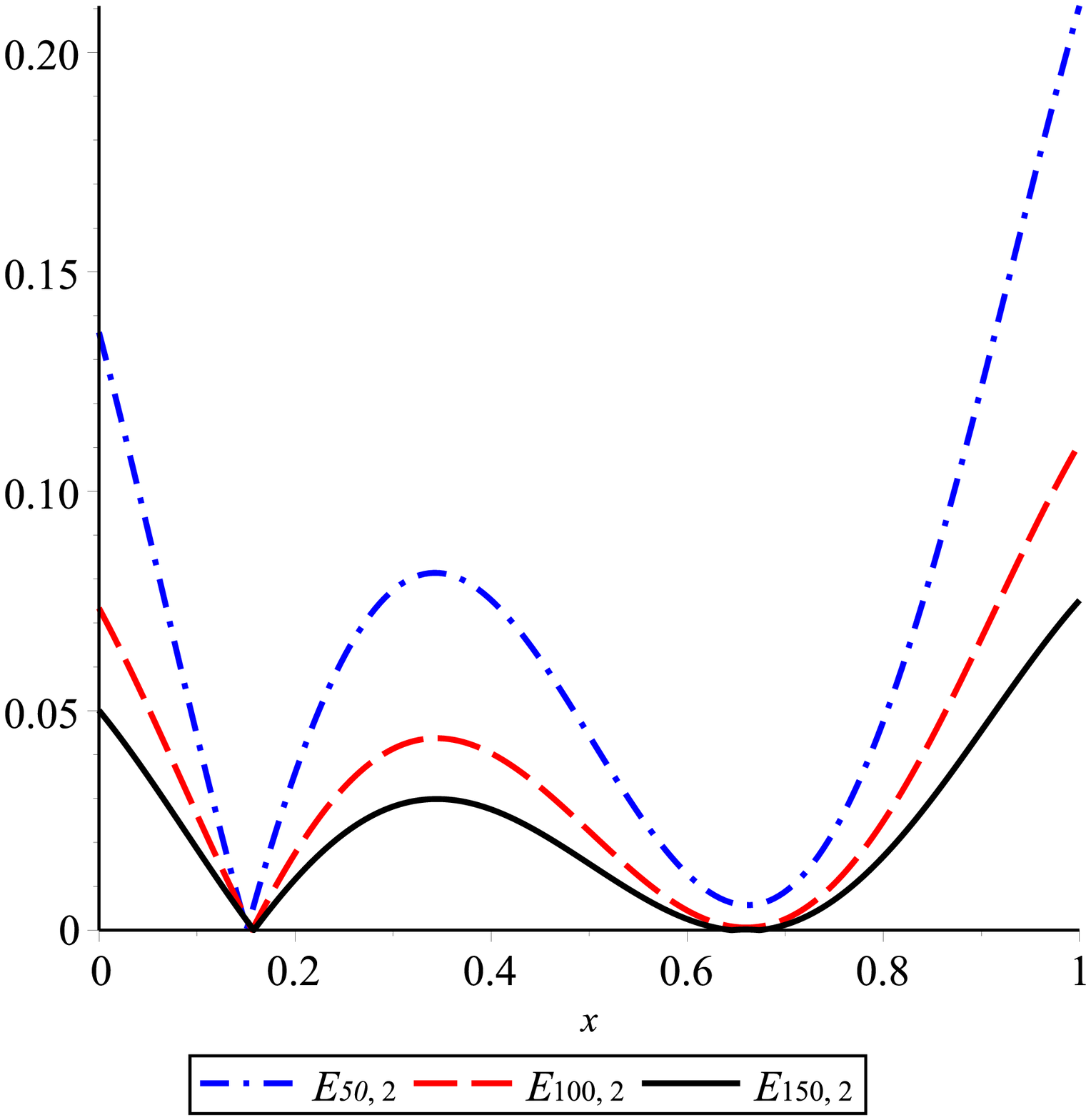}
				\caption{Error $E_{n,r}(f;x)$, for  $n\in\{50,100,150\}$}\label{fig:4}
			\end{figure}
		\end{minipage}
	\end{minipage}
\end{example}

\begin{example} Let $f(x)=\displaystyle \frac{1}{20}x^5-\frac{3}{32}x^4+\frac{13}{192}x^3-\frac{3}{128}x^2$, $r=2$ and $E_{n,r}(f;x)=$ \linebreak $\left|\left(M_n(f;x)\right)^{(r)}-M_{n-r}(f^{(r)}(x))\right|$. 
	In Figure \ref{fig:5} are given the graphs of the functions $f^{(r)}$,  $M_{n-r}(f^{(r)})$ and  $(M_nf)^{(r)}$ for $n=50$ and $r=2$.
	Also,   for $n\in\{50,100,150\}$ the absolute value of the differences are illustrated in Figure \ref{fig:6}.
	
	\begin{minipage}{\linewidth}
		\centering
		\begin{minipage}{0.4\linewidth}
			\begin{figure}[H]
				\includegraphics[width=\linewidth]{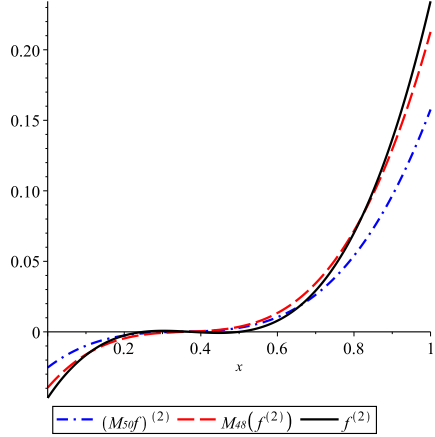}
				\caption{Approximation process by  $M_{n-r}(f^{(r)})$ and $(M_nf)^{(r)}$ }\label{fig:5}
			\end{figure}
		\end{minipage}
		\hspace{0.05\linewidth}
		\begin{minipage}{0.4\linewidth}
			\begin{figure}[H]
				\includegraphics[width=\linewidth]{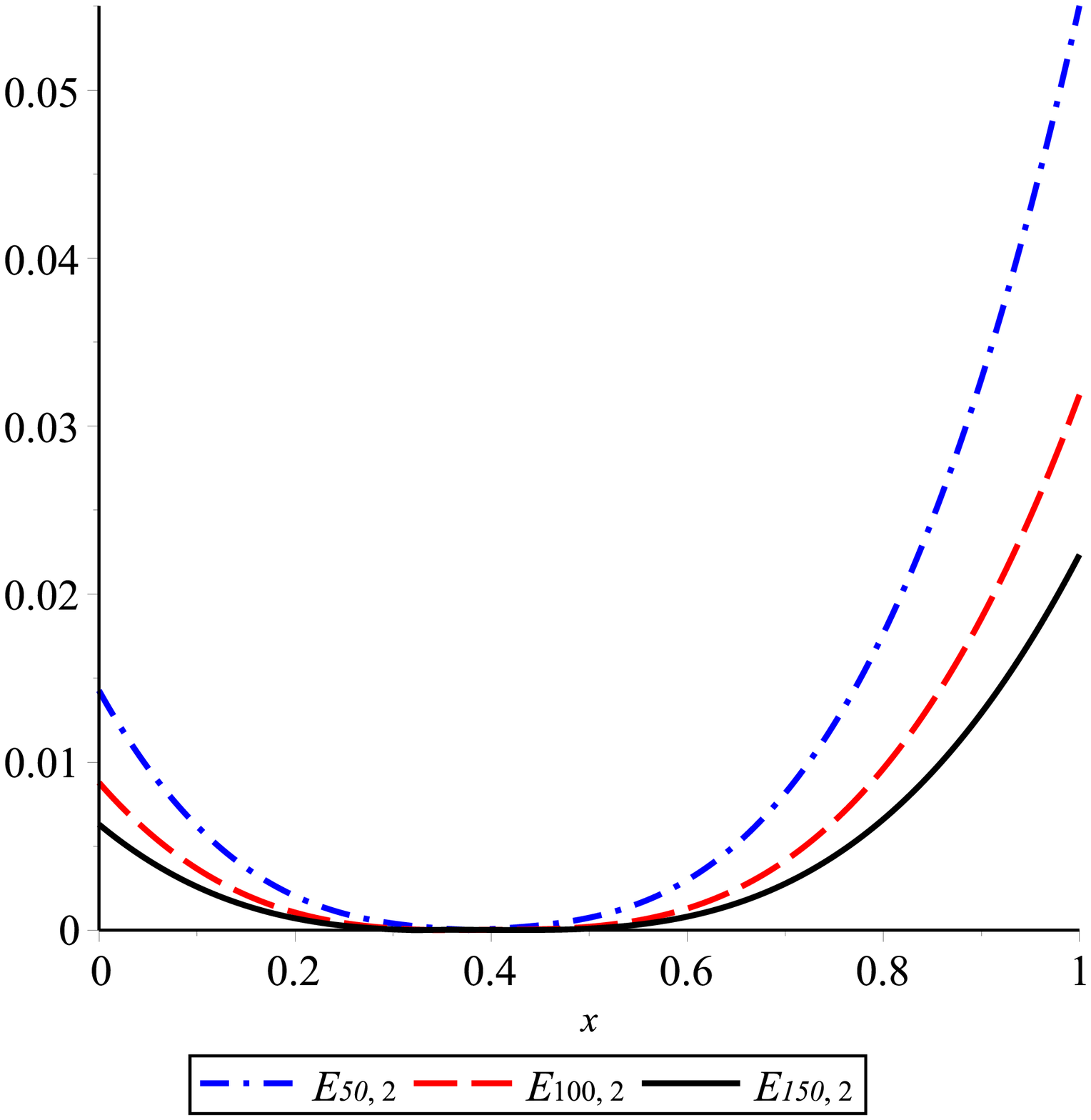}
				\caption{Error $E_{n,r}(f;x)$, for  $n\in\{50,100,150\}$}\label{fig:6}
			\end{figure}
		\end{minipage}
	\end{minipage}
\end{example}

\begin{example} Let $f(x)=\displaystyle \frac{1}{20}x^5-\frac{17}{144}x^4+\frac{7}{72}x^3-\frac{1}{32}x^2$, $r=2$ and $E_{n,r}(f;x)=$ \linebreak $\left|\left(U_n(f;x)\right)^{(r)}-U_{n-r}(f^{(r)}(x))\right|$. 
	In Figure \ref{fig:7} are given the graphs of the functions $f^{(r)}$,  $U_{n-r}(f^{(r)})$ and  $(U_nf)^{(r)}$ for $n=50$ and $r=2$.
	Also,   for $n\in\{30,40,50\}$ the absolute value of the differences  are illustrated in Figure \ref{fig:8}.
	
	\begin{minipage}{\linewidth}
		\centering
		\begin{minipage}{0.4\linewidth}
			\begin{figure}[H]
				\includegraphics[width=\linewidth]{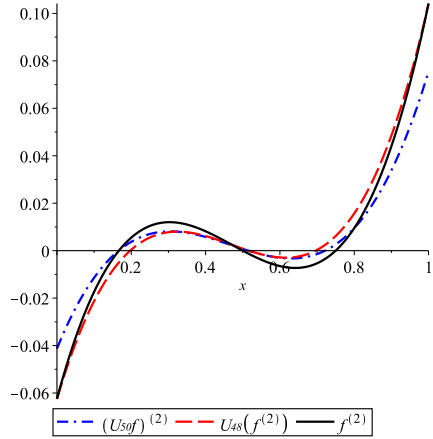}
				\caption{Approximation process by  $U_{n-r}(f^{(r)})$ and $(U_nf)^{(r)}$ }\label{fig:7}
			\end{figure}
		\end{minipage}
		\hspace{0.05\linewidth}
		\begin{minipage}{0.4\linewidth}
			\begin{figure}[H]
				\includegraphics[width=\linewidth]{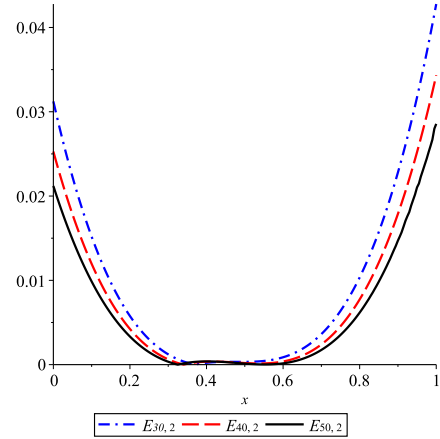}
				\caption{Error $E_{n,r}(f;x)$, for  $n\in\{30,40,50\}$}\label{fig:8}
			\end{figure}
		\end{minipage}
	\end{minipage}
\end{example}

$$    $$

\noindent{\bf Acknowledgement.} The work of the first author was financed from Lucian Blaga University of Sibiu research grants LBUS-IRG-2018-04.

$   $

\end{document}